\documentclass[12pt]{amsart}
\usepackage{amsfonts,latexsym,rawfonts,amsmath,amssymb,amsthm,times,a4wide}
\usepackage{graphicx}
\usepackage[plainpages=false]{hyperref}
\usepackage{mathrsfs}
\usepackage{enumerate}
\usepackage{bbm}

\numberwithin{equation}{section}

\usepackage{amscd}
\usepackage{eufrak}
\usepackage{euscript}
\usepackage{epsfig}
\usepackage{array}
\usepackage{enumerate}
\usepackage{color}
\usepackage{wasysym}
\usepackage{pdfsync}
\usepackage{stmaryrd}

\numberwithin{equation}{section}

\newcommand{\beq}{\begin{equation}}
\newcommand{\eeq}{\end{equation}}
\newcommand{\beqs}{\begin{eqnarray*}}
\newcommand{\eeqs}{\end{eqnarray*}}
\newcommand{\beqn}{\begin{eqnarray}}
\newcommand{\eeqn}{\end{eqnarray}}
\newcommand{\beqa}{\begin{array}}
\newcommand{\eeqa}{\end{array}}

\newtheorem{prop}{Proposition}[section]
\newtheorem{theo}[prop]{Theorem}
\newtheorem{lem}[prop]{Lemma}

\newtheorem{cor}[prop]{Corollary}
\newtheorem{rem}[prop]{Remark}

\newtheorem{defi}[prop]{Definition}

\newcommand{\tm}{\begin{theo}}
\newcommand{\tmd}{\end{theo}}
\newcommand{\co}{\begin{cor}}
\newcommand{\cod}{\end{cor}}
\newcommand{\prp}{\begin{prop}}
\newcommand{\prpd}{\end{prop}}

\newcommand{\R}{{\mathbb R}}
\newcommand{\N}{{\mathbb N}}

\newcommand{\T}{\mathcal{T}}
\newcommand{\wt}{\widetilde}

\begin{document}

\title[Combinatorial Ricci flows with applications to the hyperbolization of cusped 3-manifolds]{Combinatorial Ricci flows with applications to the hyperbolization of cusped 3-manifolds}

\author{Ke Feng}
\address{Ke Feng: School of Mathematical Sciences, University of Electronic Science and Technology of China; No.2006, Xiyuan Ave, West Hi-Tech Zone, Chengdu, Sichuan, 611731, P.R.China}
\email{youyouguzhe@126.com}

\author{Huabin Ge}
\address{Huabin Ge: School of Mathematics, Renmin University of China, Beijing, 100872, P.R. China}
\email{hbge@ruc.edu.cn}

\author{Bobo Hua}
\address{Bobo Hua: School of Mathematical Sciences, LMNS,
Fudan University, Shanghai 200433, China; Shanghai Center for
Mathematical Sciences, Fudan University, Shanghai 200433,
China.}
\email{bobohua@fudan.edu.cn}

\begin{abstract}{
In this paper, we adopt combinatorial Ricci curvature flow methods to study the existence of cusped hyperbolic structure on 3-manifolds with torus boundary.
For general pseudo 3-manifolds, we prove the long-time existence and the uniqueness for the extended Ricci flow for decorated hyperbolic polyhedral metrics.
We prove that the extended Ricci flow converges to a decorated hyperbolic polyhedral metric if and only if there exists a decorated hyperbolic polyhedral metric of zero Ricci curvature.
If it is the case, the flow converges exponentially fast. These results apply for cusped hyperbolic structure on 3-manifolds via ideal triangulation.}

%For a pseudo 3-manifold admitting decorated hyperbolic polyhedral metrics of zero curvature, we prove that the extended curvature flow $l(t)$ converges to a metric of zero curvature exponentially fast.
\end{abstract}

\maketitle
%\tableofcontents

%\begin{enumerate}
%\item Change Q to $\square.$
%\end{enumerate}

\section{Introduction}
{We continue our program to study the hyperbolization of 3-manifolds using combinatorial Ricci flow methods. In the earlier paper \cite{[FGH]}, we obtained a unique complete hyperbolic structure and a geometric triangulation on a compact 3-manifolds whose boundary components have genus at least 2, under the assumption that there is a triangulation with edge valences at least 10. In this paper, we consider
cusped hyperbolic 3-manifolds, which are noncompact complete hyperbolic 3-manifolds with cusp ends and finite volume. %Given a 3-manifold with finite cusp ends, called a cusped 3-manifold, which is the interior of a compact 3-manifold with torus boundary components.
Given a compact 3-manifold $N$ with torus boundary components, we consider the problem whether there is a cusped hyperbolic structure on $N-\partial N.$ Given an ideal triangulation for $N,$ we show that the following are equivalent: the existence of a cusped hyperbolic structure compatible with the triangulation and the convergence of the combinatorial Ricci flow.}

%we consider a 3-manifold with finite cusp ends, that is, a cusped 3-manifold which is the interior of a compact 3-manifold with torus boundary component

In the seminar work \cite{[T]}, Thurston introduced 2-dimensional circle packings to construct hyperbolic manifolds or orbifolds. Chow and Luo \cite{[BL]} introduced a combinatorial Ricci flow for triangulated surfaces. For a compact triangulated surface of non-positive Euler characteristic, they proved that the Ricci flow with any initial circle packing metric exists for all times and converges exponentially fast to Koebe-Andreev-Thurston's circle packing metrics. To extend Thurston's circle packings to higher dimensional cases, Cooper and Rivin \cite{[CR]} studied ball packings of 3-manifolds. Inspired by the work of Chow and Luo, Glickenstein \cite{[Gl1],[Gl2]} introduced a combinatorial Yamabe flow based on Euclidean triangulations defined by ball packings. % (Luo \cite{[L2]} also introduced a combinatorial Yamabe flow on surfaces).
The second author, Jiang and Shen \cite{[GJS]} studied the convergence of Glickenstein's Yamabe flow for regular ball packings in Euclidean background geometry.  The second and third authors \cite{[GH]} studied the combinatorial Yamabe flow for triangulations of closed 3-manifolds in hyperbolic background geometry and proved the convergence of the flow under some combinatorial assumption.

We recall the setting of 3-dimensional triangulated spaces.
Let $\{T_1,\cdots, T_t
\},$ $t\in \N$, be a finite collection of combinatorial tetrahedra and $\mathscr{T}$ be the disjoint union $T_1 \sqcup\cdots \sqcup T_t,$ which is a simplicial complex.
The quotient space $(M,\mathcal{T})=\mathscr{T}/\sim,$ via a family of affine isomorphisms pairing faces of
tetrahedra in $\mathscr{T}$, is called a \emph{compact pseudo 3-manifold} $M$ together with a triangulation $\mathcal{T}.$ Note that $\mathcal{T}$ consists of equivalent classes of simplexes in $\mathscr{T}.$ {Pseudo 3-manifolds are very general concepts, which include manifolds with triangulation as special cases.}
$M$ is called a \emph{closed pseudo 3-manifold} if each codimension-$1$ face of tetrahedra in $\mathscr{T}$ is identified with another codimension-$1$ face. We denote by $V=V(\T)$ (resp. $E=E(\mathcal{T})$) the set of vertices (resp. edges) in $\mathcal{T}.$ We define the \emph{valence of an edge} $e\in E,$ denoted by $d_e$, to be the number of edges in
$\mathscr{T}$ in the equivalent class of $e.$

%Introduce decorated tetrahedron. state the result of Luo Yang.

An ideal tetrahedron in $\mathbb{H}^3$ is a convex hull of $4$ points $\{v_1,v_2,v_3,v_4\}$ on $\partial\mathbb{H}^3$ in generic position. A \emph{decorated ideal tetrahedron} defined by Penner \cite{[P]} is a pair $(s, \{H_1, H_2, H_3, H_{4}\})$ where $s$ is an ideal tetrahedron and $H_i$ is an 2-horosphere centered at the vertex $v_i$ for any $1\leq i\leq 4.$ For each edge $e = v_iv_j$ in a decorated ideal tetrahedron we assign the signed length $l_{ij}$ as follows: the absolute value $|l_{ij}|$ of the length is the distance between $H_i \cap e$ and  $H_j \cap e$, $l_{ij} > 0$ if $H_i$ and $H_j$ are disjoint, and $l_{ij} \le 0$ if $H_i \cap H_j \neq \emptyset$. Codimension-$1$ faces of decorated ideal tetrahedron are decorated ideal triangles in totally geodesic hyperbolic planes.%, and $s \cap H_i$ is isometric to a Euclidean triangle.

\begin{defi}\label{def:t1} A \emph{decorated hyperbolic polyhedral metric}, \emph{decorated metric} in short, on a closed pseudo 3-manifold $(M, \T)$ is obtained by replacing each tetrahedra in $\T$ by a decorated ideal tetrahedron and replacing the affine gluing homeomorphisms by isometries preserving the decoration, i.e. gluing decorated ideal tetrahedra along codimension-$1$ faces. We denote by $l=(l(e_1), \dots, l(e_m))$, where $E = \{e_1, \dots, e_m\},$ the (signed) edge length vector of a decorated metric, and by $\mathscr{L}(M, \mathcal{T})\subset \R^E$ the set of all decorated metrics on $(M, \mathcal{T})$ parametrized by the edge length vector $l$. The \emph{Ricci curvature} $K_e(l)$ of the metric $l$ assigned to each edge $e$ is $2\pi$ minus the cone angle at the edge $e,$  i.e. the total dihedral angle surrounding the edge $e.$ We write $K(l)=(K_{e_1}(l), \dots, K_{e_m}(l)).$\end{defi}

%\begin{defi}\label{def:t1} A {hyper-ideal polyhedral metric}, called a \emph{hyper-ideal metric} in short, on $(M,\T)$ is obtained by replacing each tetrahedron in $\T$ by a hyper-ideal tetrahedron and replacing the affine gluing homeomorphisms by isometries preserving the corresponding hexagonal faces. We denote by $l\in \mathbb{R}^E_{>0}$ the edge length vector of the hyper-ideal metric, written as $l=(l(e_1), \dots, l(e_m))$, where $E = \{e_1, \dots, e_m\}.$ The above construction yields a metric space $S(M,\T,l),$ which is uniquely determined by $l.$  We denote by $\mathscr{L}(M, \mathcal{T})\subset \R^E_{>0}$ the set of all hyper-ideal metrics on $(M, \mathcal{T})$ parametrized by the edge length vector $l$. \end{defi}

The main purpose is to find decorated metrics $l$ with no singularity on edges, i.e. $K_e(l)=0$ for all $e\in E,$ {called \emph{zero-curvature decorated metrics}.} We recall the motivation for the above construction in the literature, see e.g. \cite{[L],[LY]}. Suppose that $N$ is a compact 3-manifold with boundary. {The purpose is to find a hyperbolic metric of finite volume on $N - \partial N$ with cusp ends.} Let $C(N)$ be the compact 3-space obtained by coning off each boundary component of $N$ to a point. In particular, if $N$ has $k$ boundary components, then there are exactly $k$ cone points $\{v_1,...,v_k\}$ in $C(N)$ so that $C(N) - \{v_1,...,v_k\}$ is homeomorphic to $N - \partial N.$ An \emph{ideal triangulation} $\T$ of $N$ is a triangulation $\T$
of $C(N)$ such that the vertices of the triangulation are exactly the cone points
$\{v_1,...,v_k\}.$ By Moise \cite{[M]}, every such compact 3-manifold $N$ can be ideally triangulated. In our terminology, $(C(N), \T)$ is a closed pseudo 3-manifold
 and $N$ is homeomorphic to $C(N)-st(v_1,...,v_k),$ where $st(v_1,...,v_k)$ is the
open star of the vertices $\{v_1,...,v_k\}$ in the second barycentric subdivision of the
triangulation $\T.$ {As in Definition~\ref{def:t1}, we can endow $(C(N), \T)$ with various decorated metrics. If there is a decorated metric  $l\in \mathscr{L}(C(N), \mathcal{T})$ with zero Ricci curvature at edges, then we obtain a hyperbolic metric on $N - \partial N$ with cusp ends via gluing these ideal tetrahedra. We call this a \emph{cusped hyperbolic structure} on $N$ associated with the ideal triangulation $\T.$

%This is called a \emph{geometric decomposition} (or \emph{geometric realization}) of a hyperbolic metric on $N$ associated with the ideal triangulation $\T.$}

% and use the curvature flow \eqref{eq:newflow} to find possible hyperbolic metrics on $G.$

Motivated by Chow and Luo \cite{[BL]} and Luo \cite{[L]}, for a compact 3-manifold with boundary equipped with ideal triangulation, {or more generally a closed pseudo 3-manifolds $(M,\T),$ we consider the following combinatorial Ricci flow $l(t)\in \mathscr{L}(M,\T)$  to study the existence of hyperbolic metrics,
\begin{equation}\label{eq:luoflow}\frac{d }{dt}l(t)=K(l(t)),\quad\forall t\geq0.\end{equation} This flow was first proposed by Luo \cite{[L]} in the hyper-ideal setting and further studied in the decorated setting by \cite{TianyuYang}. %The vector-valued equation \eqref{eq:luoflow} reads as $$\frac{d }{dt}l(e,t)=K_e(l(t)),\quad \forall e\in E, t\geq0.$$ }
This is a negative gradient flow of a local convex function on $\mathscr{L}(M,\T),$ related to the so-called co-volume function. The difficulty is that the domain $\mathscr{L}(M,\T)$ is not convex in $\R^E.$

To circumvent the difficulty, Luo and Yang \cite{[LY]} extended the set of decorated metrics to a general framework. A \emph{generalized decorated tetrahedron} is a (topological) tetrahedron of vertices $\{v_1, . . . , v_4\}$ so
that each edge $v_iv_j$ is assigned a real number $l_{ij}=l_{ji},$ $1\leq i\neq j\leq 4,$ called the (signed) length. %A decorated ideal
%tetrahedron (with the signed edge lengths) is a generalized decorated tetrahedron.
The space of
all generalized decorated tetrahedra parameterized by the edge length vectors $l = (l_{12} , . . . , l_{34})$ is $\R^6$.
%{Let $T$ be a tetrahedron, corresponding to a truncated tetrahedron as in Figure~\ref{fig:hyperideal1}.}
For a generalized decorated tetrahedron, Luo and Yang \cite{[LY]} defined the \emph{extended dihedral angles}, still denoted by ${\alpha}_{ij}:\R^6\to\R$, extending dihedral angles for decorated ideal tetrahedra, which turns out to be continuous functions of the edge lengths $l_{ij};$ see Section~\ref{sec:pre}. %see Definition~\ref{defi:dihe} and Section~\ref{sec:pre}.

%We define generalized hyper-ideal metrics on a compact pseudo 3-manifold $(M,\T)$ as follows. For any $l\in \R^E_{>0},$ we replace each tetrahedron in $\T$ by a generalized hyper-ideal tetrahedron with edge lengths given by $l,$ and glue them together in the topological sense. Since there are possibly some degenerate hyper-ideal tetrahedra, it may not produce any metric space structure. However, the extended dihedral angles are well defined, which are sufficient for our applications. %Note that there are possibly flat hyper-ideal tetrahedra, while the dihedral angles and volumes are still well defined, see Section~\ref{sec:pre}.}}

%\red{\begin{defi}
%We call $\widetilde l:E\to \R^{|E|}_{>0}$ a generalized hyper-ideal metric on $(M,\T).$ The curvature of an edge $e,$ $\wt{K}_e(l)$ ($\wt{K}_e$ in short), is defined similarly as in \eqref{eq:curv} by using the dihedral angles $\wt{\alpha}_{ij}.$
%\end{defi} For a generalized hyper-ideal metric on $(M,\T),$ it associates with a metric space via gluing generalized hyper-ideal tetrahedra with the edge length $l$ along codimension-1 faces. There might be some degenerate hyper-ideal tetrahedra, while the dihedral angles and volumes are well defined.}
\begin{defi}\label{defi:ghi}  Let $(M,\T)$ be a closed pseudo 3-manifold. We call any $l\in \R^E$ a \emph{generalized decorated hyperbolic polyhedral metric}, \emph{generalized decorated metric} in short, on $(M,\T),$ which replaces each tetrahedron in $\T$ by a generalized decorated ideal tetrahedron with (signed) edge lengths given by $l.$ The \emph{generalized Ricci curvature} of an edge $e,$ denoted by $\wt{K}_e(l),$ is defined similarly as in Definition~\ref{def:t1} by using extended dihedral angles on edges, and we write the generalized Ricci curvature vector as $\wt{K}(l).$ \end{defi}

We say that $l\in \R^E$ is a \emph{zero-curvature generalized decorated metric} if $\wt{K}(l)=0.$
The \emph{extended Ricci flow} $l(t)\in \R^E$ is defined as follows,
\begin{equation}\label{exkl}\frac{d }{dt}l(t)=\wt{K}(l(t)),\quad\forall t\geq0.\end{equation}
We prove the long-time existence and the uniqueness of the extended Ricci flow.
\begin{theo}\label{UNI}
For a closed pseudo 3-manifold $(M, \mathcal{T})$ and any initial generalized decorated metric $l_0\in\R^E,$ there exists a unique solution $\{l(t)|t\in [0,\infty)\}\subset \R^E$ to the extended Ricci flow \eqref{exkl}.
\end{theo}
%Note that the generalized curvature $\wt{K}$ extends the curvature $K$ for hyper-ideal metrics, i.e. $\wt{K}(l)={K}(l)$ for any $l\in \mathscr{L}(M,\T).$
The following are main results of the paper.
\begin{theo}\label{thm:main1}
For a closed pseudo 3-manifold $(M, \mathcal{T}),$ let $\{l(t)\}_{t \ge 0}$ be a solution to the extended Ricci flow \eqref{exkl}.
\begin{enumerate}
\item If there is no zero-curvature generalized decorated metrics, then $l(t)$ diverges to infinity in subsequence, i.e. there exists a subsequence $t_n\to\infty$ such that $|l(t_n)|\to \infty, n\to\infty.$
\item There is a zero-curvature decorated metric if and only if $l(t)$ converges to a zero-curvature decorated metric. If it is the case, the convergence is exponentially fast.
%\item If there is a zero curvature (smooth) decorated hyperbolic polyhedral metrics $\widetilde l$, then $\widetilde K_i(l(t))$ converges to $\widetilde K(\widetilde l) = 0$.
\end{enumerate}
\end{theo}
\begin{rem} \begin{enumerate}[(i)]
\item The first assertion gives a necessary condition for the non-existence result of zero-curvature generalized decorated metrics.
\item By the second assertion, the extended Ricci flow provides an effective method for numerically computing zero-curvature decorated metrics.
\end{enumerate}
\end{rem}

{Note that a finite cusped 3-manifold with an ideal triangulation is a closed pseudo 3-manifold. The following theorem is a consequence of Theorem~\ref{thm:main1}.
\begin{theo}\label{thm:main2}
For a finite cusped 3-manifold $N$ with an ideal triangulation $\mathcal{T}$, let $\{l(t)\}_{t \ge 0}$ be a solution to the extended Ricci flow \eqref{exkl}.
\begin{enumerate}
\item If there is no cusped hyperbolic structure on $N$ associated with $\T,$ then $l(t)$ diverges to infinity in subsequence, i.e. there exists a subsequence $t_n\to\infty$ such that $|l(t_n)|\to \infty, n\to\infty.$
\item There is a cusped hyperbolic structure on $N$ so that $\T$ is isotopic to a geometric triangulation if and only if $l(t)$ converges to a decorated metric.
Moreover, the flow converges exponentially fast in this case, and the cusped hyperbolic structure on $N$ is unique by the famous Mostow-Prasad rigidity theorem.
%\item If there is a zero curvature (smooth) decorated hyperbolic polyhedral metrics $\widetilde l$, then $\widetilde K_i(l(t))$ converges to $\widetilde K(\widetilde l) = 0$.
\end{enumerate}
\end{theo}}

%For a decorated ideal tetrahedron, its co-volume function responding to the Legendre transformation of its volume function which only related on three dihedral angles, thus we can parameterize it with the length assignment.
The key ingredient of the proof is that the extended Ricci flow is the negative gradient flow of a convex function on $\R^E,$ related to the co-volume function. The function appeared before in Cohn, Kenyon, and Propp \cite{CKP}, and Bobenko, Pinkahl and Springborn \cite{[BPB]}, and  Luo and Yang \cite{[LY]}.

A closed pseudo 3-manifold $(M,\T)$ is called \emph{edge-transitive} if for any edges $e,\hat{e}\in E,$ there exists an automorphism of the triangulation $\T$ which maps $e$ to $\hat{e}.$ This yields that $d_e=d_{\hat{e}}$ for any $e,\hat{e}\in E.$ We obtain the following corollary of Theorem~\ref{thm:main1}.
\begin{cor}\label{cor:main} Let $(M,\T)$ be an  edge-transitive closed pseudo 3-manifold, and $d=d_e$ for some $e\in E.$
\begin{enumerate}
\item If $d\neq 6,$ then there exists no zero-curvature generalized decorated metrics, and the extended Ricci flow $l(t)$
diverges to infinity in subsequence.
\item If $d=6,$ then there exists a zero-curvature decorated metric, and the extended Ricci flow
$l(t)$ converges to a zero-curvature decorated metric exponentially fast.
\end{enumerate}
\end{cor}

The paper is organized as follows. In Section \ref{sec:pre}, we collect basic facts on
the geometry of decorated ideal tetrahedra, including volume and co-volume functions, etc. In Section 3, we study the extended Ricci flow and prove the main results of the paper.

\bigskip
\noindent {\bf Acknowledgements}
We thank Yi Liu, Feng Luo, Jiming Ma, Tian Yang for many discussions on related problems in this paper. %We shared our results to  in earlier times, and we thank them for helpful comments. %We thank anonymous referees for providing many valuable comments and suggestions to improve the writing of the paper.

H. G. is supported by NSFC, no. 11871094. B. H. is supported by NSFC, no.11831004 and no. 11926313. F. K. is supported by NSFC, no. 11901009.

%\begin{rem}
%\red{Sometimes we also use $l \in \mathscr{L}(M, \mathcal{T})$ to represent one metric determined by the edge length function $l$, and even we also use $l$ to represent the edge length vector $(l(e_1), \dots, l(e_m))$ of $S(M,\T,l)$, where $E = \{e_1, \dots, e_m\}$ }.
%\end{rem}

%\begin{rem}
%\blue{delete}\red{Here, we need point out: for the glued 3-manifold is closed, the hyper-ideal polyhedral metric is not a real geometrical metric for it. If one hyperbolic metric is well-defined in the truncated tetrahedron so that it's one hyper-ideal tetrahedron in hyperbolic space, then it cannot   extend to the whole tetrahedron by the definition of hyper-ideal tetrahedron(see Section 2.1). While, if one 3-manifold with boundary is glued by truncated tetrahedra, then the hyper-ideal polyhedral metric is one real metric of it, obviously.}
%\end{rem}
%\begin{defi}\label{defi:realcurv}For a closed pseudo 3-manifold $(M,\T)$ with the edge length vector $l,$ the \emph{curvature} at each edge $e$ is defined as
%\begin{equation}\label{eq:curv}K_e(l)=2\pi-C_e\end{equation}
%where $C_e$ is the cone angle at $e,$ i.e. the total dihedral angle in tetrahedra incident to $e,$ see Definition~\ref{defi:pre}. {This provides the curvature vector,
%$K=K(l)=(K_{e_1}(l),\cdots,K_{e_m}(l)),$ where $E=\{e_1,\cdots,e_m\}.$ }
%\end{defi}

\section{Preliminaries}\label{sec:pre}

%Two edges of tetrahedra in $\T$ are called equivalent if they are mapped to the same set in $M$. We define edges in the triangulation $\mathscr{T}$ to be equivalence classes of edges in tetrahedra in $T$. We use $E = E(\mathscr{T})$ and

%\subsection{Generalized decorated ideal tetrahedra}
We recall some results on decorated ideal tetrahedra; see Luo and Yang \cite{[LY]}.
An ideal tetrahedron in $\mathbb{H}^3$ is determined up to isometry by its six dihedral angles on edges. These angles satisfy the condition that angles at opposite edges are the same and the sum of all angles is $2\pi$. Thus the set of all ideal tetrahedra modulo isometry can be identified with $ \{(a, b, c) \in \mathbb{R}^3_{> 0}| a + b + c =\pi\}$.

Let $(s, \{H_1, H_2, H_3, H_{4}\})$ be a decorated ideal tetrahedron with signed edge length $l_{ij}=l_{ji},$ where $s$ is an ideal tetrahedron $\{v_1,v_2,v_3,v_4\}$ and $H_i,$ $1\leq i\leq 4,$ are 2-horospheres centered at $v_i.$ Note that for any $i,$ $s \cap H_i$ is isometric to a Euclidean triangle.

%For a decorated ideal triangle $(s, \{H_1, H_2, H_3\})$, the length $a_i$ of the horocyclic arc in $H_i$ bounded by the two edges of $s$ from $v_i$ is called the angle at $v_i$. By Penner's cosine law, for $\{i, j, k\} = \{1, 2, 3\}$,
%\beq
%a_i = e^{(l_{jk} - l_{ij} - l_{ik})/2}.
%\eeq

%For a decorated ideal tetrahedron, we know that the dihedral angles at opposite edges are the same by the following lemma.
The following result is well-known; see e.g. \cite[Lemma~2.5]{Luo2011} or \cite[Lemma 4.2.3]{[BPB]}
\begin{lem}
Suppose that ${l_{ij}}$ are the edge lengths of a decorated ideal tetrahedron $\sigma=(s, \{H_1, H_2, H_3, H_{4}\})$. Then all four Euclidean triangles $\{s\cap H_i\}$ are similar to the Euclidean triangle $\tau$ of edge lengths $e^{(l_{ij} + l_{kh})/2},e^{(l_{ik} + l_{jh})/2},$ and $e^{(l_{ih} + l_{jk})/2},$ so that the triangle inequalities hold, i.e.
$$
e^{(l_{ij} + l_{kh})/2} + e^{(l_{ik} + l_{jh})/2} > e^{(l_{ih} + l_{jk})/2},
$$
for \{i,j,k,h\} = \{1,2,3,4\}. The dihedral angle $\alpha_{ij}$ of $\sigma$ at the edge $v_iv_j$ is equal to the inner angle of $\tau$ opposite to the edge of length $e^{(l_{ij} + l_{kh})/2}$. Conversely, if $(l_{12}, \dots, l_{34}) \in \mathbb{R}^6$ satisfies the above triangle inequalities, then there is a unique decorated ideal tetrahedron having $l_{ij}$ as the lengths. %In particular, $\alpha_{ij} = \alpha_{kh}$ for $\{i,j,k,h\} = \{1,2,3,4\}.$
\end{lem}

%By the above lemma, $\alpha_{ij} = \alpha_{kh}$, for $\{i,j,k,h\} = \{1,2,3,4\}.$ That is, dihedral angles at opposite edges are the same. Moreover, t
The set of all decorated ideal tetrahedra can be parametrized by the edge length as follows, $$\mathscr{L}:=\{(l_{12},\cdots, l_{34} ) \in \R^6: e^{(l_{ij} + l_{kh})/2} + e^{(l_{ik} + l_{jh})/2} > e^{(l_{ih} + l_{jk})/2},\ \{i,j,k,h\}\ \mathrm{distinct}\}.$$ %Thus, we can talk about the dihedral angle of a quad in a decorated ideal tetrahedron.
Luo and Yang \cite{[LY]} extended it to the set of generalized decorated ideal tetrahedra $\{(l_{12},\cdots, l_{34} ) \in \R^6\}.$ A generalized decorated ideal tetrahedron is characterized by the length vector $(l_{12},\cdots, l_{34} ) \in \R^6.$  To define dihedral angles and the volume of a generalized decorated tetrahedron, we need to use the notion of generalized Euclidean triangles and their angles. A generalized Euclidean triangle is a (topological) triangle of vertices $\{v_1, v_2, v_3\}$ so that each edge is assigned a positive number; let $x_i$ be the assigned length of the edge $v_jv_k$ where $\{i, j, k\} = \{1, 2, 3\}$. The inner angle $a_i = a_i(x_1, x_2, x_3)$ at the vertex $v_i$ is the inner angle of the Euclidean triangle of edge lengths $x_1, x_2, x_3$ opposite to the edge of length $x_i$, if the triangle inequalities hold; or $a_i = \pi, a_j = a_k = 0$, if $x_i \ge x_j + x_k$.

\begin{lem}{(Luo \cite{[L1]})}\label{F}
The angle function $a_i(x_i, x_j, x_k): \mathbb{R}^3_{> 0} \to [0, \pi]$ is continuous so that $a_1 + a_2 + a_3 = \pi$ and the $C^0$ - smooth differential 1-form $\sum^3_{i = 1} a_id(\ln x_i)$ is closed on $\mathbb{R}^3_{> 0}$. Furthermore, for $u_i = \ln x_i$, the integral $F(u) = \int^u_0 \sum^3_{i = 1} a_idu_i$ is a $C^1$ - smooth convex function in $(u_1, u_2, u_3)$ on $\mathbb{R}^3$ so that $F$ is strictly convex when restricted to $\{u \in \mathbb{R}^3 | u_1 + u_2 + u_3 = 0, e^{u_i} + e^{u_j} > e^{u_k}\}$ and $F(u + (k, k, k)) = F(u) + k\pi$ for all $k \in \mathbb{R}$.
\end{lem}

For a generalized decorated tetrahedron of length vector $l = (l_{12}, \dots, l_{34}) \in \mathbb{R}^6$, the dihedral angle $\alpha_{ij}$ at the edge $v_iv_j$ is defined to be the inner angle of the generalized Euclidean triangle of edge lengths $e^{(l_{ij} + l_{kh})/2}$, $e^{(l_{ik} + l_{jh})/2}$ and $e^{(l_{ih} + l_{jk})/2}$ so that $\alpha_{ij}$ is opposite to the edge of length $e^{(l_{ij} + l_{kh})/2}$ for $i, j, k, h$ distinct. Hence the dihedral angles $\alpha_{ij}=\alpha_{kh},$ $\{i,j,k,h\} = \{1,2,3,4\},$ are continuous functions on $\R^6.$

We study the properties of volume functions and co-volume functions on generalized decorated ideal tetrahedra.
Recall that the Lobachevsky function is defined as $\Lambda(x)=\int_0^x \ln |2\sin t| dt.$
The volume of a generalized decorated tetrahedron $\{v_1,v_2,v_3,v_4\}$ of lengths $l_{ij}$, denoted by $vol(l)$, is defined to be $\frac{1}{2}\sum_{i < j} \Lambda (\alpha_{ij}(l))$, where $\alpha_{ij}$ is the extended dihedral angle at edge $e_{ij} = v_iv_j$. By this definition, $vol(l)$ is the hyperbolic volume of the underlying ideal tetrahedron if $l\in \mathscr{L};$ $vol(l)=0,$ otherwise. The co-volume of a generalized decorated tetrahedron of lengths $l_{ij}$ is defined as
$$cov(l):=2vol(l) + \sum_{1\leq i<j\leq 4}\alpha_{ij} l_{ij}.$$

\begin{prop}[\cite{[LY]}]\label{cov}
The co-volume function $cov: \mathbb{R}^6 \to \mathbb{R}$ is a $C^1$ - smooth convex function so that for any $1\leq i<j\leq 4,$
\beq
\frac{\partial cov(l)}{\partial l_{ij}} = \alpha_{ij}.
\eeq
%when $\alpha_i$ is the dihedral angle at the ith edge.
\end{prop}

For any $\alpha\in \R^6$ satisfying $\alpha_{ij}=\alpha_{kh}$ and $\alpha_{ij}+\alpha_{ik}+\alpha_{ih}=\pi$ for $\{i,j,k,l\}=\{1,2,3,4\},$ we define
$$F_{\alpha}:\R^6\to \R,\quad l\mapsto F_{\alpha} (l)= cov(l) - \alpha \cdot l.$$
The group $\R^4$ acts on $\R^6$ as follows: for any $w\in\R^4, l\in \R^6,$
$$(w+l)_{ij}:= l_{ij}+w_i+w_j,\quad \forall\  1\leq i<j\leq 4.$$ Then one can show that
$F_\alpha$ is invariant under the action, i.e.
$$F_\alpha(w+l)=F_{\alpha}(l),\quad \forall\ l\in \R^6, w\in \R^4.$$
We denote by $\widehat{\R^4}:=\{w+0: w\in \R^4\}$ with $0\in \R^6,$ which is a 4-dimensional linear subspace of $\R^6,$ and by $\R^6/\widehat{\R^4}$ the orthogonal subspace of $\widehat{\R^4}$ w.r.t. the standard inner product. We prove the following result.
\begin{lem}\label{lem:convex} The co-volume function $cov: \mathbb{R}^6 \to \mathbb{R}$ is smooth on
$\mathscr{L}.$ The restriction function $cov|_{\R^6/\widehat{\R^4}}$ is strictly convex in $\mathscr{L}\cap \R^6/\widehat{\R^4}.$ In fact, for any $l\in \mathscr{L},$ the kernel of $\mathrm{Hess}(cov)|_l$ is $\widehat{\R^4}$ and $\mathrm{Hess}(cov)|_l$ is positive definite in $\R^6/\widehat{\R^4}.$%For any $\alpha\in \R^6$ satisfying $\alpha_{ij}=\alpha_{kh}$ and $\alpha_{ij}+\alpha_{ik}+\alpha_{ih}=\pi$ for $\{i,j,k,l\}=\{1,2,3,4\},$ $F_\alpha$ is strictly convex on $\R^6/\hat{\R}^4.$
\end{lem}
\begin{proof} For $\alpha=(\frac{\pi}{3},\frac{\pi}{3},\cdots,\frac{\pi}{3}),$ consider the function $F_\alpha$ on $\R^6.$ For any $l\in \R^6,$ since $F_\alpha$ is invariant under the action of $\R^4,$
$\widehat{\R^4}$ is contained in the kernel of $\mathrm{Hess}(F_\alpha)|_l=\mathrm{Hess}(cov)|_l.$ By \cite[Theorem~2.14]{TianyuYang}, the rank of $\mathrm{Hess}(cov)|_l$ is two. Since $cov$ is convex by Proposition~\ref{cov}, the kernel of $\mathrm{Hess}(cov)|_l$ is $\widehat{\R^4}.$ This implies that $\mathrm{Hess}(cov)|_l$ is positive definite in $\R^6/\widehat{\R^4}.$ The result follows.
\end{proof}

%So, if the generalized decorated tetrahedron $\sigma$ is  a decorated ideal tetrahedron, $vol(l)$ is the hyperbolic volume of the underlying ideal tetrahedron; if $\sigma$ is not a decorated ideal tetrahedron, $vol(l) = \Lambda(0) + \Lambda(0) + \Lambda(\pi) = 0$. Thus, here we parameterize the space of all ideal tetrahedra by $\mathcal{A} = \{(a, b, c) \in \mathbb{R}^3_{> 0} | a + b + c = \pi\}$, the volume function $vol$ defined on $\mathcal{A}$ is given by $vol(l) = \Lambda(a) + \Lambda(b) + \Lambda(c)$.

%From this proposition we know that $\psi_\theta$ is one proper convex function, and its minimal value exists. This proposition is proved by Luo and Yang, readers could find the whole proof in \cite{[LY]}.

%Take a finite disjoint collection $T$ of Euclidean tetrahedra and identify some of the codimension 1 faces in $T$ in pairs by affine homeomorphisms. The quotient space $(M, T)$ is a compact pseudo 3-manifold $M$ together with a triangulation $\mathscr{T}$.
Now we consider the properties of generalized decorated metric on pseudo 3-manifolds with triangulation.
Let $(M,\T)=\mathscr{T}/\sim$ be a closed pseudo 3-manifold,  where $\mathscr{T}$ be the disjoint union of tetrahedra $T_1 \sqcup\cdots \sqcup T_t.$ We denote by  $T=T(\mathcal{T})$ the set of tetrahedra in $\mathscr{T}.$
A $quad$ in the triangulation $\T$ is a pair of opposite edges in a tetrahedron in $T$ and we write $\square = \square(\T)$ for the set of all quads in $\T$. For $q \in \square$, $e \in E$ and $\sigma \in T$, we use $q \subset \sigma$ to denote that the quad $q$ is contained in the tetrahedron $\sigma$ and use $q \sim e$ or $e \sim q$ to denote that $e \subset q$.
For any $\sigma\in T$ and any generalized decorated metric $l,$ we denote by $l_\sigma\in \R^6$ the restriction of the metric $l$ on $\sigma.$
%Now, suppose $(M, \T)$ is a triangulated compact pseudo 3-manifold with the sets of edges $E = E(\T)$ and quads $Q = Q(\T)$.

\begin{defi}
An \emph{angle assignment} on $(M, \T)$ is a map: $\alpha: \square \to \mathbb{R}_{\ge 0}$ so that for each tetrahedron $\sigma \in T$, $\sum_{q \subset \sigma} \alpha(q) = \pi$. The cone angle of $\alpha$ is defined to be $k_\alpha: E \to \mathbb{R}_{\ge 0}$ where $$k_\alpha(e) = \sum_{q \sim e} \alpha(q).$$
For a generalized decorated metric $l$ on $(M, \T),$
it associates with an angle assignment $\alpha = \alpha_l:$ for any $q\in \square,$
$\alpha(q)$ is defined as the extended dihedral angle at an edge $e$ such that $e\sim q$ in the tetrahedron with generalized decorated metric induced by $l.$ We write $k_l = k_{\alpha_l}$ for the cone angles at edges.
%And the curvature map related to $\alpha$ is defined to be $K: E \to \mathbb{R}$ where $K(e) = 2\pi - k_\alpha(e)$.
%The volume function of $\alpha$, denoted by $vol(\alpha)$, is defined to be $vol(\alpha) = \sum_{q \in Q} \Lambda(\alpha(q))$.

\end{defi} Using the above notation, the generalized Ricci curvature is given by
$$\wt{K}_e(l)=2\pi-k_l(e),\quad \forall e\in E.$$

\begin{defi}For a generalized decorated metric $l$ on $(M, \T),$ the volume (resp. co-volume) function for the metric $l$ is defined as
$$vol(l)=\sum_{\sigma \in T} vol_\sigma(l)\quad (\mathrm{resp.}\ cov(l)=\sum_{\sigma \in T} cov_\sigma(l)),$$ where $vol_\sigma(l)=vol(l_\sigma)$ (resp. $cov_\sigma(l)=cov(l_\sigma)$) is the volume (resp. co-volume) of the tetrahedron $\sigma$ endowed with the metric $l_\sigma.$
\end{defi}
Hence,
$$vol(l) = \sum_{q \in \square} \Lambda(\alpha_l(q)),$$
\beq
cov(l) =2vol(l) + l\cdot k_l= \sum_{\sigma \in T} \sum_{q \subset \sigma} \big [2\Lambda(\alpha_l(q)) + \alpha_l(q) \sum_{e \sim q} l(e)\big].
\eeq

%Further, we can give an angle assignment $\alpha = \alpha_l$, where $\alpha_l$ is the general dihedral angle function decided by $l$, since we can define the dihedral angle of $l$ at an edge $e$ in a tetrahedron $\sigma \supset e$ to be the corresponding dihedral angles in the generalized decorated tetrahedron whose edge lengths are given by $l$. The cone angle of $\alpha$ is called the cone angle of the metric $l$. We denote it by $k_l = k_{\alpha_l}$.

%The volume of a generalized decorated metric $l \in \mathbb{R}^E$ is defined to be the volume of its dihedral angle assignment, that is $vol(l) = \sum_{q \in Q} \Lambda(\alpha(q))$. The co-volume of $l$ denoted by $cov(l)$, is defined as $cov(l) = 2vol(l) + l\cdot k_l$, where $k_l$ is the cone angle of $l$ and $u \cdot v$ is the standard inner product in $\mathbb{R}^E$. By the definition of cone angle, we have
The action of $\R^V$ on $\R^E$ is defined as follows: for any $w\in\R^V, l\in \R^E,$
$$(w+l)_{e}:= l_{e}+w(e_+)+w(e_-),\quad \forall e\in E,$$ where $e_+$ and $e_-$ are end-vertices of $e.$ We denote by $\widehat{\R^V}:=\{w+0: w\in \R^V\}$ with $0\in \R^E,$ and by $\R^E/\widehat{\R^V}$ the orthogonal subspace of $\widehat{\R^V}$ w.r.t. the standard inner product.
%$$(w+l)_{ij}:= l_{ij}+w_i+w_j,\quad \forall\ l\in \R^6, 1\leq i<j\leq 4.$$
For any $l_0\in \R^E,$ we set
\begin{equation}\label{def:F}F_{l_0}(l):=cov(l)-k_{l_0} \cdot l,\quad \forall l\in \R^E.\end{equation}
Luo and Yang proved the following proposition.
\begin{prop}[Proposition~3.4 in \cite{[LY]}]\label{Prop:pp1}
\begin{enumerate}
\item For any $l_0\in \R^E,$ $F_{l_0}$ is invariant under the action of $\R^V$, i.e.
$$F_{l_0}(w+l)=F_{l_0}(l),\quad \forall\ l\in \R^E, w\in \R^V.$$
\item For any $l_0\in \mathscr{L}(M,\T),$
\begin{equation}\label{eq:proper}
\lim_{l\in \R^E/\widehat{\R^V}, l\to \infty} F_{l_0}(l)=+\infty.
\end{equation}
\end{enumerate}
\end{prop}

It was proved by Luo and Yang \cite{[LY]} that the co-volume function $cov: \mathbb{R}^E \to \mathbb{R}$ is $C^1$-smooth and convex, which is smooth on $\mathscr{L}(M,\T).$ We prove the following lemma.
\begin{lem}\label{lem:strict}
The restriction function $cov|_{\R^E/\widehat{\R^V}}$ is strictly convex in $\mathscr{L}(M,\T)\cap \R^E/\widehat{\R^V}.$
\end{lem}
\begin{proof}
Consider $F_0:\R^E\to \R$ where $0\in \R^E.$ For any $l\in \R^E,$
$$Hess(F_0)|_l=Hess(cov)|_l.$$ Since $F_0$ is invariant under the action of $\R^V,$
$\widehat{\R^V}$ is contained in the kernel of $Hess(cov)|_l.$ Since $cov$ is convex on $\R^E,$ it suffices to prove that the kernel of $Hess(cov)|_l$ is exactly $\widehat{\R^V},$ which implies that $Hess(cov)|_l$ is positive definite on $\R^E/\widehat{\R^V}.$ As a consequence, $cov|_{\R^E/\widehat{\R^V}}$ is strictly convex in $\mathscr{L}(M,\T)\cap \R^E/\widehat{\R^V}.$

Take any $x\in \R^E$ in the kernel of $Hess(cov)|_l.$ We need to show that $x\in \widehat{\R^V}.$
\begin{eqnarray*}0&=&x^T(Hess(cov)|_l) x=x^T\left(\sum_{\sigma\in T}Hess(cov_\sigma)|_l\right) x\\
&=&\sum_{\sigma\in T}x_\sigma^T\left(Hess(cov)|_{l_\sigma}\right) x_\sigma\geq 0,
\end{eqnarray*} where $x_\sigma$ is the restriction of $x$ to $\sigma.$ Hence $x_\sigma$ is in the kernel of $Hess(cov)|_{l_\sigma}.$ By Lemma~\ref{lem:convex}, there exists
$w_\sigma\in \R^4,$ a function on vertices of $\sigma$, such that
$$x_\sigma=w_\sigma+0,\quad 0\in \R^6.$$ By Neumann's Lemma \cite{Neumann1992}, see also \cite[p.1354]{Choi2004} or \cite[Lemma~3.3]{[LY]}, $w_\sigma,$
$\sigma\in T,$ are consistent, i.e. for two tetrahedra $\sigma_1,\sigma_2$ sharing a vertex $v,$ $w_{\sigma_1}(v)=w_{\sigma_2}(v).$ Hence, there exists a function $w\in \R^V$ such that the restriction of $w$ to $\sigma$ is $w_\sigma.$ So that $x=w+0, \ 0\in \R^V$ and $x\in \widehat{\R^V}.$ This proves the result.
%To prove the claim, it suffices to show that for any two tetrahedra $\sigma_1,\sigma_2$ sharing one face $f,$ then  In fact, the claim follows from the above result for a chain of tetrahedra connecting the  for which the consestitive tetrahedra sharing one face.

%Suppose that $\sigma_1$ and $\sigma_2$ correspond to the tetrahedra $\{v_1,v_2,v_3,v_4\}$ and $\{v_1',v_2',v_3',v_4'\}$ in $\mathscr{T},$ and the face $f$ corresponds to $\{v_1,v_2,v_3\}$ and $\{v_1',v_2',v_3'\}$ respectively. Note that
%\[\left\{\begin{array}{c}x(v_1v_2)=x(v_1'v_2')=w_{\sigma_1}(v_1)+w_{\sigma_1}(v_2)=w_{\sigma_2}(v_1')+w_{\sigma_2}(v_2'),\\
%x(v_2v_3)=x(v_2'v_3')=w_{\sigma_1}(v_2)+w_{\sigma_1}(v_3)=w_{\sigma_2}(v_2')+w_{\sigma_2}(v_3'),\\
%x(v_3v_1)=x(v_3'v_1')=w_{\sigma_1}(v_3)+w_{\sigma_1}(v_1)=w_{\sigma_2}(v_3')+w_{\sigma_2}(v_1').
%\end{array}\right.\] This implies that $$w_{\sigma_1}(v_i)=w_{\sigma_2}(v_i'),\quad\forall 1\leq i\leq 3.$$ This proves the claim. The lemma follows.
%By Lemma~\ref{lem:convex}, by reducing to each tetrahedron, we can prove that $Hess(cov)|_l$ has the kernel $\widehat{\mathbb{R}^V},$ and it is positive definite on $\R^E/\widehat{\R^V}.$ This proves the result.
\end{proof}

%Now, for one tetrahedron $\sigma$ and $k \in \mathbb{R}^6$, let $F_{\sigma, k} (l)= cov_\sigma(l) - k \cdot l$, since the property of $\psi_\theta$ in Proposition(\ref{phi}), we know that $F_{\sigma, k}(\widetilde l) = F_{\sigma, k}(l)$, if $(\widetilde l_i + \widetilde l_{i + 3}) - (l_i + l_{i + 3}) = c$ , $i = 1, 2, 3$ and any constant number $c \in \mathbb{R}$. This means that $F_\sigma$ is a convex function defined on the quotient space $\mathbb{R}^{E(\sigma)}/\mathbb{R}^{V(\sigma)}$. In fact, we consider the vector space $\mathbb{R}^V$ acts linearly on $\mathbb{R}^E$ by
%\beq
%(u+l)(vv') = u(v) + u(v') + l(vv'),
%\eeq
%where $u \in \mathbb{R}^V, l \in \mathbb{R}^E, (u + l) \in \mathbb{R}^E$ and the edge $vv'$ has vertices $v$, $v'$. And we will identify $\mathbb{R}^V$ with the linear subspace $\mathbb{R}^V + 0$ of $\mathbb{R}^E$.
%Then, we can easily notice that the null space of $Hess_l (F_{\sigma, k})$ is isomorphic to $\mathbb{R}^V$.

\section{The combinatorial Ricci flow}

Let $(M,\T)=\mathscr{T}/\sim$ be a closed pseudo 3-manifold.
Let $E=E(\T)=\{e_1,e_2,\cdots, e_m\},$ where $m$ is the number of edges. To simplify the notation, we write $E=\{1,2,\cdots, m\},$ that is, each edge $e_i$ is replaced by the index $i.$ In this way, each sub-index of an edge $e_i$ is replaced by $i.$ For example, we write the edge length vector $l=(l_1,l_2,\cdots l_m)$ and the Ricci curvature ${K}(l)=({K}_1(l), {K}_2(l),\cdots, {K}_m(l)),$ etc.
%This yields the curvature map $$\wt{K}:\R^E_{>0}\to \R^E,\quad l\mapsto \wt{K}(l).$$ If $l$ is a hyper-ideal metric, i.e., $l\in \mathcal{L}(M,\T),$ we write $K$ as above instead of $\wt{K}.$
%\subsection{The combinatorial curvature flow}

The combinatorial Ricci flow \eqref{eq:luoflow} in $\mathscr{L}(M, \T)$ can be written as follows:
\begin{equation*}
\begin{cases}%\label{kl}
\frac{d l_i(t)}{d t} = K_i(l(t)) ,\quad \forall i\in E, t\geq 0,\\
l(0) = l_0,
\end{cases}
\end{equation*}
where $l_0\in\mathscr{L}(M, \T)$ and $l(t)\in \mathscr{L}(M, \T),\forall t>0.$ Since $\mathscr{L}(M, \T)$ is an open subset in $\R^E$ and $K(l)$ is smooth on $\mathscr{L}(M, \T),$ Picard's theorem in ordinary differential equations yields the following.
\tm
For a closed pseudo 3-manifold $(M, \mathcal{T}),$ for any initial data $l_0\in\mathscr{L}(M, \T),$ the solution $\{l(t)|t\in [0, T )\}\subset \mathscr{L}(M, \T)$ to the combinatorial Ricci flow \eqref{eq:luoflow} exists and is unique on the maximal existence interval $[0, T )$ with $0 < T\leq \infty.$

%For one 3-manifold $(M, \mathcal{T})$ with negative Euler characteristic boundary, where $l_0 \in \mathscr{L}(M, \T)$ is a fixed hyper-ideal polyhedral metric, suppose $\{u(t)|t \in [0, T)\}$ is the unique maximal solution to flow:
%\beq
%\begin{cases}\label{old}
%\frac{\partial l_i(t)}{\partial t} = K_i(l)\cdot l_i,\\
%l(0) = l_0
%\end{cases}
%\eeq
%with $0 < T \le +\infty$, we can always extend it to a solution $\{l(t) | t \in [0, +\infty)\}$, when $T < +\infty$.
\tmd

%\textcolor{red}{
%\begin{rem}
%In the decorated setting, one has a discrete version of Gauss-Bonnet formula $\sum_i K_i=0$, hence we may assume $\sum_i l_i=0$ along this flow.
%\end{rem}}

We define the functional $$H:\mathscr{L}(M, \T)\to \R,\quad l\mapsto H(l)=cov(l) - 2\pi\sum_{i=1}^m l_i.$$  By Proposition~\ref{cov},
\beqs
\frac{\partial H}{\partial l_j}
&=& \frac{\partial cov(l)}{\partial l_j} - 2\pi=\sum_{\sigma\in T}\frac{\partial cov_\sigma(l)}{\partial l_j} - 2\pi\\
%&=& 2\frac{\partial vol}{\partial \alpha_k}\frac{\alpha_k}{\partial l_j} +
%\sum_{\sigma \in T}(\alpha_\sigma)_j +  \sum_k \big (\sum_{\sigma \in T}\frac{\partial (\alpha_\sigma)_k}{\partial l_j}\big )l_k -2\pi\\
&=& \sum_{ q\sim j, q\in \square}\alpha(q) -2\pi\\
&=& -K_j
\eeqs
By this result, $l\in \mathscr{L}(M,\T)$ has zero Ricci curvature if and only if $l$ is a critical point of the functional $H.$

Hence, the combinatorial Ricci flow in \eqref{eq:luoflow} can be written as
\begin{equation*}\begin{cases}
\frac{\partial l_i(t)}{\partial t} = -\partial_i H(l),\\
l(0) = l_0.
\end{cases}
\end{equation*}
This implies that the combinatorial Ricci flow is the negative gradient flow of the functional $H.$

\begin{prop}\label{H}
The functional $H$ is non-increasing along the combinatorial Ricci flow (\ref{eq:luoflow}), i.e. for any solution $l(t)$ to the flow (\ref{eq:luoflow}),
\[\frac{\partial H(l(t))}{\partial t} \le 0.\]
\end{prop}
\begin{proof}
By direct calculation,
\[\frac{\partial H(l(t))}{\partial t} = -|K(l(t))|^2 \le 0.\]
\end{proof}
%In fact, it could be seen as one negative gradient flow by some parameter variation.
\begin{prop}\label{prop:zero}
Let $l(t)$ be a solution to the combinatorial Ricci flow (\ref{eq:luoflow}) which converges to $\bar l \in \mathscr{L}$. Then
$K(\bar l) = 0$.
\end{prop}
\begin{proof}
This is well-known in classical ODE theory. For the convenience of readers, we include the proof here. By Proposition \ref{H}, $H(l(t))$ is non-increasing. Moreover,
$\{H(l(t)) : t \ge 0\}$ is bounded, since $H$ is continuous on $\mathscr{L}$ and $l(t) \to \bar l$, $t \to \infty$.
Hence the following limit exists and is finite,
\[\lim_{t \to \infty} H(l(t)) = C.\]
Consider the sequence $\{H(l(n))\}^\infty_{n=1}$. By the mean value theorem, for any
$n \ge 1$ there exists $t_n \in (n, n + 1)$ such that
\beq
H(l(n + 1)) - H(l(n)) =\left.\frac{d}{dt} \right|_{t = t_n} H(l(t)) = - |K(l(t_n))|^2.
\eeq

Note that $\lim_{n \to \infty} H(l(n + 1)) - H(l(n)) = 0$. We get
\[\lim_{n \to \infty} K_i(l(t_n)) = 0, \quad \forall i\in E.\]

Since $l(t_n) \to \bar l$ as $n \to \infty$, the continuity of $K_i$ yields that $K_i(\bar l) = 0$ for any $i\in E$.
This proves the proposition.
\end{proof}

Next we consider the extended Ricci flow on the set of generalized decorated metrics, i.e. $\R^E.$ The extended Ricci flow \eqref{exkl} can be written as
\begin{equation*}
\begin{cases}
\frac{\partial l_i(t)}{\partial t} = \widetilde K_i(l(t)),\quad \\
l(0) = l_0,
\end{cases}
\end{equation*} where $l_0\in\R^E$ and $l(t)\in \R^E,\forall t>0.$

We define the extended functional $\widetilde H$ on $\mathbb{R}^E$ by
\[\widetilde H(l) = cov(l) - 2\pi\sum_{i=1}^m l_i.\]
Note that $\widetilde H$ is a $C^1$- smooth convex functional on $\R^E$ and it extends the functional $H$ on $\mathscr{L}(M,\T).$  Direct calculation yields that $$\frac{\partial \widetilde H}{\partial l_i} = -\widetilde K_i,\quad \forall i\in E.$$
Hence, $l\in\R^E$ has zero generalized Ricci curvature if and only if it is a critical point of the functional $\wt{H}.$ As in Proposition~\ref{H}, one can show that the functional $\wt{H}$ is non-increasing along the extended Ricci flow \eqref{exkl}.

Note that if a closed pseudo 3-manifold $(M,\T)$ supports a zero-curvature generalized decorated metric $\hat{l},$ then
$$\wt{H}(l)=F_{\hat{l}}(l),$$ where $F_{\hat{l}}$ is defined in \eqref{def:F}. In particular,
$\wt{H}$ is invariant under the action of $\R^V$ on $\R^E$ by Proposition~\ref{Prop:pp1}.

\begin{prop}\label{prop:inv} Let $(M,\T)$ be a closed pseudo 3-manifold which has a zero-curvature generalized decorated metric. Then the solution of the extended Ricci flow is compatible with the action $\R^V$ on $\R^E,$ i.e.,
for any initial data $l_0\in \R^E$ with $l_0=w+l_0^\top,$ where $w\in \R^V$ and $l_0^\top$ is the projection of $l_0$ to $\mathbb{R}^{E}/ \widehat{\mathbb{R}^{V}},$ then the solution $l(t)$ of the extended Ricci flow with the initial data $l_0$ satisfies
$$l(t)=w+l^\top(t), \ \forall t\in[0,\infty),$$ where $l^\top(t)$ is the solution of  the extended Ricci flow with the initial data $l_0^\top.$ Moreover, $l^\top(t)\in \mathbb{R}^{E}/ \widehat{\mathbb{R}^{V}}$ for all $t\in[0,\infty).$

\end{prop}
\begin{proof} Note that $\wt{H}$ is invariant under the action of $\R^V$ on $\R^E$ and the extended Ricci flow is the negative gradient flow of $\wt{H}.$ By these observation, for any initial data $l_0^\top\in \mathbb{R}^{E}/ \widehat{\mathbb{R}^{V}},$ there is a solution of the extended Ricci flow $l^\top(t)\in \mathbb{R}^{E}/ \widehat{\mathbb{R}^{V}}$ for all $t\in[0,\infty).$ In fact, it is induced by the negative gradient flow of $\wt{H}|_{\mathbb{R}^{E}/ \widehat{\mathbb{R}^{V}}}$ on $\mathbb{R}^{E}/ \widehat{\mathbb{R}^{V}}.$ By the uniqueness of the solution of the extended Ricci flow, we prove the above results.

\end{proof}

%\subsection{Properness of \texorpdfstring{$\widetilde H$}{}}
\begin{theo}[\cite{[LY]}]\label{thm:l1}
Let $(M,\T)$ be a closed pseudo 3-manifold which has a zero-curvature decorated metric. Then $$\lim_{l\in\mathbb{R}^{E}/ \widehat{\mathbb{R}^{V}},l \to \infty} \widetilde H(l) = +\infty.$$ In particular, $\widetilde H(l)$ is proper on $\mathbb{R}^{E}/ \widehat{\mathbb{R}^{V}}.$
\end{theo}
\begin{proof} By $\wt{H}(l)=F_{l_0}(l),$ where $K(l_0)=0,$ the results follow from Proposition~\ref{Prop:pp1}.
\end{proof}

The following rigidity result was proved by Luo and Yang.
\begin{theo}[Theorem~1.2 in \cite{[LY]}]\label{thm:rigidity}
Let $(M,\T)$ be a closed pseudo 3-manifold which has a zero-curvature decorated metric $l_0\in \mathscr{L}(M,\T).$ Then
the set of zero-curvature generalized decorated metrics on $(M,\T)$ is unique up to the action $\R^V$ on $\R^E,$ i.e.
$$\{l\in \R^E: \wt{K}(l)=0\}=\{w+l_0: w\in \R^V\}.$$
\end{theo}
\begin{proof} Here we give an alternative proof using the strict convexity of the functional $\wt{H}.$ Note that $\wt{H}$ is invariant under the action of $\R^V$ on $\R^E.$ If $l\in \R^E$ has zero Ricci curvature, then so does $w+l$ for any $w\in \R^V.$ Let $l_0^\top$ be the projection of $l_0\in \mathscr{L}(M,\T)$ to $ \R^E/\widehat{\R^V}.$ Then $l_0^\top\in \mathscr{L}(M,\T)$ and it also has zero Ricci curvature.
Note that zero-curvature generalized decorated metrics are critical points of the functional $\wt{H}.$ Since $\wt{H}$ is convex, all of them are in fact minimizers.  By Lemma~\ref{lem:strict},  $\wt{H}$ is strictly convex in $\mathscr{L}(M,\T)\cap \R^E/\widehat{\R^V}.$ Since $l_0^\top\in \mathscr{L}(M, \T)\cap \R^E/\widehat{\R^V}$ is a minimizer of $\wt{H},$ it is the unique critical point of $\wt{H}$ on $\R^E/\widehat{\R^V}.$ This proves the result. %By Theorem~\ref{thm:l1}, $\wt{H}$ is proper and bounded from below. So that there is a minimizer of $\wt{H}.$

%We consider the functional $H:\R^E_{>0}\to \R.$ By Proposition~\ref{prop:stc}, it is $C^1$-smooth and convex on $\R^E_{>0},$ and is smooth and strictly convex on $\mathscr{L}(M, \T).$ By \eqref{eq:deriv}, the metrics in $\R^E_{>0}$ with zero curvature correspond to the critical points of the functional $H.$ Note that any critical point of a convex function on a convex domain is a minimizer. So that the zero-curvature metric $l$ is a minimizer of $H.$ Moreover, $l\in \mathscr{L}(M, \T)$ and $H$ is strictly convex on $\mathscr{L}(M, \T).$ This yields that the minimizer, i.e. the critical point, of $H$ is unique on $\R^E_{>0}.$ We prove the result.
\end{proof}

Next, we prove the long-time existence and uniqueness of the extended Ricci flow.
%\begin{theo}\label{thm:extended}
%For a closed pseudo 3-manifold $(M, \mathcal{T})$ and any initial data $l_0\in\R^E,$ there exists a unique solution $\{l(t)|t\in [0,\infty)\}\subset \R^E$ to the extended Ricci flow \eqref{exkl}.
%\end{theo}
\begin{proof}[Proof of Theorem~\ref{UNI}]
Since the extended Ricci curvature $\widetilde K(l)$ is continuous on $\mathbb{R}^E,$ by Peano's existence theorem in classical ODE theory, the extended Ricci flow \eqref{exkl} has at least one solution on some interval $[0, T),$ with the maximal existence time $T.$ Furthermore, $|\wt{K}(l)|$ are uniformly bounded by a constant $C > 0$ which depends on the triangulation $\T.$ Hence
\[|l(t)| \le |l(0)| + Ct\]
for all $t \ge 0$, which implies that $T=\infty.$

Next, we prove the uniqueness of the solution of the extended Ricci flow.
Let $l(t)$ and $\hat l(t)$ be two solutions of of the extended Ricci flow with initial data $l_0.$ We need to prove that $l(t) = \hat l(t)$ for any $t \in [0, \infty).$ %$\widetilde K_i = -\frac{\partial \widetilde H}{\partial l_i}$ and $\widetilde H$ is $C^1$-smooth convex by the extending. Thus,
Set $$h(t) = |l(t) - \hat l(t)|^2 \ge 0, \quad t \in [0, \infty).$$

Then,
\beqs
h'(t) &=& 2(l(t) - \hat l(t))\cdot(\widetilde K(l(t)) - \widetilde K(\hat l(t)))\\
      &=& -2(l(t) - \hat l(t))\cdot(\nabla \widetilde H(l(t)) - \nabla \widetilde H(\hat l(t)))\\
      &\leq&0,
\eeqs where the last inequality follows from the convexity of a $C^1$ functional $\wt{H}.$ %Hence
For $h(0) = 0$, $h(t) =0$, i.e. $l(t) = \hat l(t),$ for any $t\in [0,\infty).$ This proves the result.

%Since $\widetilde H$ is a $C^1$ convex function,
%\[\widetilde H(x) \ge \widetilde H(y) + \nabla \widetilde H(y) \cdot (x - y).\]

\end{proof}

By the same argument as in Proposition~\ref{prop:zero}, one can prove the following.
\begin{prop}
If a solution $l(t)$ of the extended Ricci flow \eqref{exkl} converges to some $\widetilde l \in \mathbb{R}^E$ as $t \to +\infty$. Then $\widetilde K(\widetilde l) = 0.$
\end{prop}
%\begin{proof}
%For every $i$, the component $\widetilde K_i$ is continuous on $\mathbb{R}^N$, so $\widetilde K_i(l(t))$ converge to $\widetilde K_i(\widetilde l)$ as $t \to +\infty$.
%On the other hand, there is a sequence ${\xi_n}\uparrow +\infty$, such that $l'(\xi_n) = l(n+1) - l(n) \to 0$. This implies that every $K_i(l(\xi_n)) = l'(\xi_n) \to 0$.
%Thus, $\widetilde K(\widetilde l) = 0$.
%\end{proof}

We are ready to prove the main result.
%\begin{theo}
%If there do not exist $\widetilde l$ satisfying $\widetilde K(\widetilde l) = 0$. Then $\lim_{t \to +\infty} |l(t)| = +\infty$.
%\end{theo}
\begin{proof}[Proof of Theorem~\ref{thm:main1}]
We prove the first assertion.
Suppose that it is not true, then there exists a constant $C$ such that
$|l(t)|\leq C,\ t\in [0,\infty).$ %So that there is a subsequence of $\{t_n\},$ still denoted by $\{t_n\}$ and $\hat l\in \R^E,$ such that $ l(t_n) \to \hat l, \ n\to\infty$.

%Let $\phi(t) = \widetilde H(l(t))$, then $\phi$ is one $C^1$ function and $\phi'(t)= \nabla \widetilde H \cdot \widetilde K = - |\widetilde K|^2 \le 0$,

Let $\phi(t) = \widetilde H(l(t)).$ Then $\phi(t)$ is bounded on $t\in[0,\infty).$ Since $\phi(t)$ is non-increasing,  the following limit exists
$$\lim_{t\to\infty} \phi(t)=C\in \R.$$ %Then $\phi$ is a $C^1$ function and $\phi'(t)= \nabla \widetilde H \cdot \widetilde K = - |\widetilde K|^2 \le 0$, since $\widetilde H \in C^1(\mathbb{R}^N)$ and $l(t)$ satisfies the extended Ricci flow equation.
%By the continuity, $\phi(t_n) \to \widetilde H(\hat l),$ $n\to \infty.$
By the mean value theorem, there are $\xi_n\in (n,n+1)$ such that $$\phi'(\xi_n) = \phi(n + 1) - \phi(n) \to 0,\quad n\to\infty$$
Since $\phi'(t)= - |\widetilde K(l(t))|^2,$ $$\wt K(l(\xi_n))\to 0,\quad n\to\infty.$$
By passing to a subsequence, $l(\xi_n)\to l_\infty\in \R^E.$ This implies that $$\wt{K}(l_\infty)=0.$$ This yields a contradiction and proves the first assertion.
%With $\phi'(t) \le 0$, we can get $\phi'(t) \to 0$ as $t \to +\infty$, that implies $\widetilde K_i(t) \to 0$. While, by the continuity of $\widetilde K$, we have $\widetilde K(l(t_n)) \to \widetilde K(\hat l)$. Hence, $\widetilde K(\hat l) = 0$, which contradicts with the initial condition.

Now we prove the second assertion. By Proposition~\ref{prop:zero}, we only need to prove that if there is a zero-curvature decorated metric, then $l(t)$ converges to a zero-curvature decorated metric exponentially fast. By Proposition~\ref{prop:inv}, it suffices to prove the result for the initial data $l_0\in \mathbb{R}^{E}/ \widehat{\mathbb{R}^{V}}.$ Hence the solution $l(t)$ of the extended Ricci flow  satisfies $l(t)\in \mathbb{R}^{E}/ \widehat{\mathbb{R}^{V}}$ for all $t\in[0,\infty).$

By Luo and Yang's rigidity theorem, Theorem~\ref{thm:rigidity}, there is a unique zero-curvature decorated metric $\hat{l}$ in
$\mathbb{R}^{E}/ \widehat{\mathbb{R}^{V}}.$
By Theorem~\ref{thm:l1}, the functional $\wt{H}|_{\mathbb{R}^{E}/ \widehat{\mathbb{R}^{V}}}$ is proper on $\mathbb{R}^{E}/ \widehat{\mathbb{R}^{V}}$ and $\wt{H}$ is bounded from below.

%Based on the argument before, we know that $l'(t) = -\nabla\widetilde H$ is a negative gradient flow.

For $\phi(t) = \widetilde H(l(t)),$ since $\phi(t)$ is non-increasing, the following limit exists
$$\lim_{t\to\infty} \phi(t)=C\in \R.$$ By the properness of $\wt{H}|_{\mathbb{R}^{E}/ \widehat{\mathbb{R}^{V}}},$ $\{l(t):t\in [0,\infty)\}$ is contained in a compact subset of $\mathbb{R}^{E}/ \widehat{\mathbb{R}^{V}}.$
By the mean value theorem, there is a sequence $\xi_n\in (n,n+1)$ such that
$$-|\wt{K}(l(\xi_n))|^2=\phi'(\xi_n) = \phi(n + 1) - \phi(n) \to 0,\quad n\to\infty.$$
Passing to a subsequence, still denoted by $\xi_n,$
$$l(\xi_n)\to l_\infty,\quad n\to\infty.$$ By the continuity of $\wt{K},$ $\wt{K}(l_\infty)=0.$ By
Theorem~\ref{thm:rigidity}, we have $$l_\infty=\hat{l}.$$ Hence for any neighbourhood $U$ of $\hat{l}$ in $\mathbb{R}^{E}/ \widehat{\mathbb{R}^{V}},$ for sufficiently large $\xi_n,$
$$l(\xi_n)\in U.$$ By Lemma~\ref{lem:strict},
$Hess (\wt{H}|_{\mathbb{R}^{E}/ \widehat{\mathbb{R}^{V}}})_{\hat{l}}$ is positive definite, which implies the critical point $\hat{l}$ is a local attractor of the extended Ricci flow restricted on $\mathbb{R}^{E}/ \widehat{\mathbb{R}^{V}}.$ By Lyapunov's theorem in the ODE theory, the flow $l(t)$ converges to $\hat{l}$ exponentially fast for any initial data. This proves the result.

\end{proof}

{ Now we prove Theorem~\ref{thm:main2}.
\begin{proof}[Proof of Theorem~\ref{thm:main2}] This follows from Theorem~\ref{thm:main1}.
The other statements are well-known results in hyperbolic geometry.

\end{proof}}

In the following we prove Corollary~\ref{cor:main}.
\begin{proof}[Proof of Corollary~\ref{cor:main}] We claim that if there exists a zero-curvature generalized decorated metric $l,$ then there exists a zero-curvature decorated metric $\hat{l}$ with constant length, i.e. $\hat{l}_{i}=\hat{l}_{j}$ for any $i,j\in E.$ Since $(M,\T)$ is edge-transitive and $l=(l_1,\cdots, l_m)$ has zero Ricci curvature, for any permutation $\eta\in S_m,$
$l_\eta:=(l_{\eta(1)}, \cdots, l_{\eta(m)})$ has zero Ricci curvature. Since the set of zero-curvature generalized decorated metrics $K$ consists of critical points of the convex functional $\wt{H},$ $K$ is a convex set in $\R^E.$ Since $\hat{l}:=(\frac{1}{m}\sum_i l_i,\cdots,\frac{1}{m}\sum_i l_i)$ is in the convex hull of $\{l_\eta\}_{\eta\in S_m},$ $\hat{l}\in K.$ This proves the claim.

For the assertion (1), $d_i=d\neq 6$ for any $i\in E.$ We argue by contradiction. Suppose that there is a zero-curvature generalized decorated metric $l,$ then by the claim there exists a zero-curvature decorated metric $\hat{l}$ with constant length. Hence for each edge $e$ and quad $q$ with $e\sim q,$ $\alpha(q)=\frac{\pi}{3}.$ This implies that
$$K_e(\hat{l})=2\pi-d\frac{\pi}{3}\neq 0.$$ This is a contradiction. Hence there exists no zero-curvature generalized decorated metrics, and the extended Ricci flow $l(t)$
diverges to infinity in subsequence by Theorem~\ref{thm:main1}.

For the assertion (2), $d_i=6$ for any $i\in E.$ One can verify that $l=(1,\cdots,1)\in \mathscr{L}(M,\T)$ is a zero-curvature decorated metric. Then result follows from Theorem~\ref{thm:main1}.

\end{proof}

Next, we consider the prescribed Ricci curvature problem.
The curvature map is defined as
\begin{equation}\label{eq:curv}K:\mathscr{L}(M, \mathcal{T})\to \R^E,\ l\mapsto K(l).\end{equation}
General questions are as follows:
For $\overline{K}\in\R^E,$ is there any $l\in \mathscr{L}(M, \mathcal{T})$ such that $K(l)=\overline{K}?$ If there are some, how could one find the set $K^{-1}(\overline{K})?$ The structure of $K^{-1}(\overline{K})$ is determined by Luo and Yang's rigidity theorem,  Theorem~\ref{thm:rigidity}. Namely, if there is $l_0\in \mathscr{L}(M, \mathcal{T})$ such that $K(l_0)=\overline{K},$ then $$K^{-1}(\overline{K}):=\{w+l_0: w\in \R^V\}.$$ For the case $\overline{K}=0,$ it reduces to the problem of zero Ricci curvature as before. To study the general problem, we introduce the following {extended Ricci flow} $l(t)\in \R^E$
\begin{equation*}\label{exkl1}\frac{d }{dt}l(t)=\wt{K}(l(t))-\overline{K},\quad\forall t\geq0.\end{equation*} Note that if there is $l_0\in \mathscr{L}(M, \mathcal{T})$ of zero Ricci curvature, then it is a negative gradient flow of the convex functional
$F_{l_0}.$ Similar results as Theorem~\ref{UNI} and Theorem~\ref{thm:main1} can be derived accordingly. We omit the proofs here.

Define the discrete Laplace operator $\Delta=-\frac{\partial K}{\partial l}$. One may also consider the following combinatorial Calabi flow
$$\frac{dl}{dt}=\Delta K.$$
which is the negative gradient flow of the combinatorial Calabi energy $C=\|K\|^2/2$. Using similar methods as in \cite{[Gthesis],[G]}, one can study the convergence of the flow and the existence of zero-curvature decorated metrics.

At the end, to realize our program about the hyperbolization of 3-manifolds by combinatorial Ricci flow methods, we need further to prove that suitable topological conditions (such as  incompressible, atoroidal) are equivalent to the convergence of the extended combinatorial Ricci flow.

%\item If (a) is true, what is the structure of $K^{-1}(\overline{K})?$
%\item If (a) is true, how could one find $l\in K^{-1}(\overline{K})?$
%\end{enumerate}
%\end{q}
%For any $a\in \R^V,$ we define the action of $a$ on $\R^E$ by
%$$(a+l)(e)=l(e)+a(e_+)+a(e_-),\quad \forall e\in E,$$ where $e_+,e_-$ are end-vertices of $e.$
%The part $(b)$ in the question is related to the rigidity of the solutions of $K(l)=\overline{K},$ which was settled by Luo and Yang \cite{[LY]} using variational methods: if there is $l\in \mathscr{L}(M, \mathcal{T})$ such that $K(l)=\overline{K},$ then $$K^{-1}(\overline{K}):=\{w+l: w\in \R^V, l\in \R^E\}.$$ %and $\mathbb{R}^{E}/ \widehat{\mathbb{R}^{V}}$ is unique.

%\begin{equation}\label{eq:luoflow}\frac{d }{dt}l(t)=K(l(t))-\overline{K},\quad\forall t\geq0.\end{equation}
%

%\begin{theo}\label{thm:main1}
%Given one pseudo 3-manifold with a triangulation $(M, \mathscr{T})$, for $\overline{K}\in \R^E,$ let $\{l(t)\}_{t \ge 0}$ be a solution to the extended Ricci flow.
%\begin{itemize}
%\item If there is no generalized decorated metrics such that $\wt{K}(l)=\overline{K},$ then $l(t)$ diverges to infinity, i.e. $|l(t)|$ converges to $+\infty$.
%\item If there is a decorated metric $\widetilde l$ such that ${K}(l)=\overline{K},$ then $l(t)$ converges to $\widetilde l$ exponentially fast.
%\item If there is a zero curvature (smooth) decorated hyperbolic polyhedral metrics $\widetilde l$, then $\widetilde K_i(l(t))$ converges to $\widetilde K(\widetilde l) = 0$.
%\end{itemize}
%\end{theo}

\bibliography{decorate}

\begin{thebibliography}{Luo11b}

\bibitem[BPS15]{[BPB]}
Alexander~I. Bobenko, Ulrich Pinkall, and Boris~A. Springborn.
\newblock Discrete conformal maps and ideal hyperbolic polyhedra.
\newblock {\em Geom. Topol.}, 19(4):2155--2215, 2015.

\bibitem[Cho04]{Choi2004}
Young-Eun Choi.
\newblock Positively oriented ideal triangulations on hyperbolic
  three-manifolds.
\newblock {\em Topology}, 43(6):1345--1371, 2004.

\bibitem[CKP01]{CKP}
Henry Cohn, Richard Kenyon, and James Propp.
\newblock A variational principle for domino tilings.
\newblock {\em J. Amer. Math. Soc.}, 14(2):297--346, 2001.

\bibitem[CL03]{[BL]}
Bennett Chow and Feng Luo.
\newblock Combinatorial {R}icci flows on surfaces.
\newblock {\em J. Differential Geom.}, 63(1):97--129, 2003.

\bibitem[CR96]{[CR]}
Daryl Cooper and Igor Rivin.
\newblock Combinatorial scalar curvature and rigidity of ball packings.
\newblock {\em Math. Res. Lett.}, 3(1):51--60, 1996.

\bibitem[FGH20]{[FGH]}
Ke~Feng, Huabin Ge, and Bobo Hua.
\newblock Convergence of curvature flows for hyper-ideal polyhedral metrics.
\newblock {\em arXiv:2009.03731}, 2020.

\bibitem[Ge12]{[Gthesis]}
Huabin Ge.
\newblock Combinatorial methods and geometric equations.
\newblock {\em PhD Thesis, Peking University}, 2012.

\bibitem[Ge18]{[G]}
Huabin Ge.
\newblock Combinatorial {C}alabi flows on surfaces.
\newblock {\em Trans. Amer. Math. Soc.}, 370(2):1377--1391, 2018.

\bibitem[GH20]{[GH]}
Huabin Ge and Bobo Hua.
\newblock 3-dimensional combinatorial {Y}amabe flow in hyperbolic background
  geometry.
\newblock {\em Trans. Amer. Math. Soc.}, 373(7):5111--5140, 2020.

\bibitem[GJS18]{[GJS]}
Huabin Ge, Wenshuai Jiang, and Liangming Shen.
\newblock On the deformation of ball packings.
\newblock {\em arXiv:1805.10573.}, 2018.

\bibitem[Gli05a]{[Gl1]}
David Glickenstein.
\newblock A combinatorial {Y}amabe flow in three dimensions.
\newblock {\em Topology}, 44(4):791--808, 2005.

\bibitem[Gli05b]{[Gl2]}
David Glickenstein.
\newblock A maximum principle for combinatorial {Y}amabe flow.
\newblock {\em Topology}, 44(4):809--825, 2005.

\bibitem[Luo05]{[L]}
Feng Luo.
\newblock A combinatorial curvature flow for compact 3-manifolds with boundary.
\newblock {\em Electron. Res. Announc. Amer. Math. Soc.}, 11:12--20, 2005.

\bibitem[Luo11a]{Luo2011}
Feng Luo.
\newblock A note on complete hyperbolic structures on ideal triangulated
  3-manifolds.
\newblock In {\em Topology and geometry in dimension three}, volume 560 of {\em
  Contemp. Math.}, pages 19--26. Amer. Math. Soc., Providence, RI, 2011.

\bibitem[Luo11b]{[L1]}
Feng Luo.
\newblock Rigidity of polyhedral surfaces, {III}.
\newblock {\em Geom. Topol.}, 15(4):2299--2319, 2011.

\bibitem[LY18]{[LY]}
Feng Luo and Tian Yang.
\newblock Volume and rigidity of hyperbolic polyhedral 3-manifolds.
\newblock {\em J. Topol.}, 11(1):1--29, 2018.

\bibitem[Moi52]{[M]}
Edwin~E. Moise.
\newblock Affine structures in {$3$}-manifolds. {V}. {T}he triangulation
  theorem and {H}auptvermutung.
\newblock {\em Ann. of Math. (2)}, 56:96--114, 1952.

\bibitem[Neu92]{Neumann1992}
Walter~D. Neumann.
\newblock Combinatorics of triangulations and the {C}hern-{S}imons invariant
  for hyperbolic {$3$}-manifolds.
\newblock In {\em Topology '90 ({C}olumbus, {OH}, 1990)}, volume~1 of {\em Ohio
  State Univ. Math. Res. Inst. Publ.}, pages 243--271. de Gruyter, Berlin,
  1992.

\bibitem[Pen87]{[P]}
R.~C. Penner.
\newblock The decorated {T}eichm\"{u}ller space of punctured surfaces.
\newblock {\em Comm. Math. Phys.}, 113(2):299--339, 1987.

\bibitem[Thu79]{[T]}
William~P. Thurston.
\newblock {\em Geometry and topology of 3-manifolds}, volume~42.
\newblock Lecture Notes, Princeton University,
  http://www.msri.org/publications/books/gt3m/, 1979.

\bibitem[Yan19]{TianyuYang}
Tianyu Yang.
\newblock A combinatorial curvature flow for ideal triangulations.
\newblock {\em PhD thesis, The University of Melbourne,
  https://minerva-access.unimelb.edu.au/handle/11343/222445}, 2019.

\end{thebibliography}
\bibliographystyle{alpha}

\end{document}